\documentclass[11pt]{amsart}
\usepackage{xcolor,mathtools}
\usepackage[OT2,T1]{fontenc}
\usepackage{enumerate}
\DeclareSymbolFont{cyrletters}{OT2}{wncyr}{m}{n}
\DeclareMathSymbol{\Sha}{\mathalpha}{cyrletters}{"58}

\include{package}
\definecolor{grn}{rgb}{0,0.6,0}
\definecolor{mrn}{rgb}{0.3,0,0}
\definecolor{blue}{rgb}{0,0,0.7}
\definecolor{Mygray}{rgb}{0.75,0.75,0.75}
\definecolor{auburn}{rgb}{0.43, 0.21, 0.1}
\definecolor{britishracinggreen}{rgb}{0.0, 0.26, 0.15}
\definecolor{taupe}{rgb}{0.28, 0.24, 0.2}
\usepackage{amsmath,amssymb}
\newtheorem{theorem}{Theorem}[section]

\newtheorem{proposition}{Proposition}[section]
\newtheorem{cor}{Corollary}[section]
\newtheorem{lemma}{Lemma}[section]

\newtheorem{remark}{Remark}[section]

\newtheorem*{ack}{Acknowledgements}

\usepackage{latexsym,amssymb}
\usepackage{amssymb}
\usepackage{graphicx}
\usepackage{ textcomp }
\usepackage{tikz}
\usetikzlibrary{shapes.geometric, arrows}
\tikzstyle{startstop} = [rectangle, rounded corners, minimum width=3cm, minimum height=1cm,text centered, draw=black, fill=white!30]
\tikzstyle{io} = [trapezium, trapezium left angle=70, trapezium right angle=110, minimum width=3cm, minimum height=1cm, text centered, draw=black, fill=white!30]
\tikzstyle{process} = [rectangle, minimum width=2cm, minimum height=1cm, text centered, draw=black, fill=white!30]
\tikzstyle{decision} = [rectangle, minimum width=1cm, minimum height=1cm, text centered, draw=black, fill=white!30]
\tikzstyle{arrow} = [thick,->,>=stealth]
\usepackage[left=2cm,right=2cm,top=2cm,bottom=2cm]{geometry}

\allowdisplaybreaks

\begin{document}
\baselineskip=14.5pt
\title[On rank and $2$-Selmer group of a family of elliptic curves]{On the Mordell-Weil rank and $2$-Selmer group of a family of elliptic curves}

\author{Pankaj Patel, Debopam Chakraborty and Jaitra Chattopadhyay}
\address[Pankaj Patel and Debopam Chakraborty]{Department of Mathematics, BITS-Pilani, Hyderabad campus, Hyderabad, INDIA}
\address[Jaitra Chattopadhyay]{Department of Mathematics, Siksha Bhavana, Visva-Bharati, Santiniketan - 731235, West Bengal, India}

\email[Pankaj Patel]{p20200452@hyderabad.bits-pilani.ac.in}

\email[Debopam Chakraborty]{debopam@hyderabad.bits-pilani.ac.in}

\email[Jaitra Chattopadhyay]{jaitra.chattopadhyay@visva-bharati.ac.in; chat.jaitra@gmail.com}

\begin{abstract}
We consider the parametric family of elliptic curves over $\mathbb{Q}$ of the form $E_{m} : y^{2} = x(x - n_{1})(x - n_{2}) + t^{2}$, where $n_{1}$, $n_{2}$ and $t$ are particular polynomial expressions in an integral variable $m$. In this paper, we investigate the torsion group $E_{m}(\mathbb{Q})_{\rm{tors}}$, a lower bound for the Mordell-Weil rank $r({E_{m}})$ and the $2$-Selmer group ${\rm{Sel}}_{2}(E_{m})$ under certain conditions on $m$. This extends the previous works done in this direction, which are mostly concerned only with the Mordell-Weil ranks of various parametric families of elliptic curves.
\end{abstract}

\renewcommand{\thefootnote}{}



\footnote{\emph{We confirm that all the data are included in the article.}}

\renewcommand{\thefootnote}{\arabic{footnote}}
\setcounter{footnote}{0}

\maketitle

\section{introduction}

Number theory is primarily concerned with Diophantine equations and their integral or rational solutions. It is difficult, in general, to determine whether a given Diophantine equation has any solution at all or not. A famous example is that of Fermat's Last Theorem which remained unresolved for more than three hundred years before Wiles (\cite{tw} and \cite{w}) proved the existence of no integral solutions using very sophisticated techniques of modern mathematics. In 1900 Hilbert asked the question, famously known as the ``Hilbert's tenth problem", whether there exists an algorithm that can decide within finitely many steps if a given Diophantine equation has solutions in $\mathbb{Z}$. In 1970, Matiyasevich \cite{mat} answered this negatively. Along a similar line, Mordell observed that the arithmetic behaviour of the points on a curve is quite closely related to the {\it genus} of it and conjectured that a curve over $\mathbb{Q}$ of genus at least $2$ can have at most finitely many rational points. This conjecture of Mordell was confirmed assertively in 1983 by Faltings.

\smallskip

Among the class of Diophantine equations, elliptic curves occupies a central position and the study of their rational points has been an important theme among number theorists. An elliptic curve $E$ over $\mathbb{Q}$ is an equation of the form $y^{2} = f(x)$, where $f(X) \in \mathbb{Q}[X]$ is a cubic polynomial having distinct roots in $\mathbb{C}$, together with a rational point $\mathcal{O}$. The set of rational points on $E$ is denoted by $E(\mathbb{Q})$ and one can define a binary operation, called addition of points on $E(\mathbb{Q})$, that makes $E(\mathbb{Q})$ into an abelian group. It is a fundamental result in the theory of the arithmetic of elliptic curves that $E(\mathbb{Q})$ is a finitely generated abelian group and thus by the structure theorem of finitely generated abelian groups, we have $$E(\mathbb{Q}) \simeq E(\mathbb{Q})_{\rm{tors}} \oplus \mathbb{Z}^{r},$$ for some integer $r \geq 0$. Here $E(\mathbb{Q})_{\rm{tors}}$ is called the torsion group of $E$ and the integer $r$, also referred to as $r(E)$, is called the Mordell-Weil rank of $E$. 

\smallskip

The computation of the Mordell-Weil ranks of elliptic curves is an important area of research in number theory due to its influence in several problems seemingly unrelated to elliptic curves. A celebrated example of such nature is the \text{congruent number elliptic curve} $E: y^{2} = x^{3} - n^{2}x$ for positive integers $n$. The rank $r(E)$ determines whether $n$ can be realized as the area of a rational-sided right-angled triangle or, equivalently, the existence of a triplet in arithmetic progression with common difference $n$, where all three terms are perfect squares (cf. \cite{con}). In \cite[Theorem 1]{bro}, Brown and Myers studied the curve $E: y^{2} = x^{3}-x+m^{2}$ over $\mathbb{Q}$ and proved that the curve has a trivial torsion group, and the Mordell-Weil rank is at least $2$. The precise statement is as follows. 
\begin{theorem}[cf. \cite{bro}, Theorem 1]
    Let $m \geq 0$ be an integer, and let $E_{m}$ be the elliptic curve with equation $y^{2} = x^{3} - x + m^{2}$. Then the following hold.
    \begin{enumerate}
    \item If $m \geq 1$, then $E_{m}(\mathbb{Q})_{\rm{tors}} = \{\mathcal{O}\}$. 
    
    \item If $m \geq 2$, the $r(E_{m}(\mathbb{Q})) \geq 2$, with $P = (0,m)$ and $Q = (-1,m)$ being two independent points. 
    
    \item There are infinitely many values of $m$ for which $r(E_{m}(\mathbb{Q})) \geq 3$. 
    \end{enumerate}
\end{theorem}

Later on, through the works of various authors (cf. \cite{ant}, \cite{cha}, \cite{eki}, \cite{nara}, \cite{juy}, \cite{rout}), the Mordell-Weil group of certain variants of the aforementioned elliptic curve were explored. In a series of two papers (cf. \cite{tad1} and \cite{tad2}), Tadic studied similar elliptic curves over function fields. In this article, we delved into a somewhat more general elliptic curve $E: y^{2} = x (x-n_{1})(x-n_{2}) + t^{2}$ for certain integers $n_{1}, n_{2}$, and $t$, and look into both the Mordell-Weil rank and the $2$-Selmer rank of the same curve. We first specify certain choices for $n_{1}, n_{2}$ and $t$ in the following remark.  
\begin{remark}\label{mainrmk}
    We choose an even integer $m$ such that $m \pm 1$ are twin primes, and $m^{2}+1$ is square-free. For every such integer $m$, we denote $n_{1} = (m^{2}+1)^{2}, n_{2} = -(m^{2}-1)^{2}$, and $t = 2m(m^{4}-1)$. This gives a specific representation for the elliptic curve $y^{2} = x(x - n_{1})(x - n_{2}) + t^{2}$, which is suitable for both the Mordell-Weil rank computation as well as the $2$-Selmer rank computation. The representation is as follows.  
    $$E_{m}: y^{2} = x(x-(m^{2}+1)^{2})(x + (m^{2}-1)^{2})+ (2m(m^{4}-1))^{2} = (x - (m^{4}-1))(x + (m^{4}-1))(x - 4m^{2}).$$ The discriminant of $E_{m}$ is a divisor of $2^{6} \cdot (m^{4}-1)^{2} \cdot (m^{4}-1 - 4m^{2})^{2} \cdot (m^{4}-1 +4m^{2})^{2}$. 
\end{remark}

%Following a similar method to the works mentioned above, one can obtain that rank $(E) \geq 2$ here also, with $E(\mathbb{Q})_{\text{tors}} \cong \mathbb{Z}/ 2\mathbb{Z} \times \mathbb{Z}/ 2\mathbb{Z}$. However, as evident from the following table, there are certain patterns in the rank and prime factors of the terms $(m^{4}-1-4m^{2}, m^{4}-1, m^{4}-1+4m^{2})$ in the determinant of $E$. 

\begin{table}
    \centering
    \begin{tabular}{|c|c|c|c|c|c|}
    \hline
        $m$ & $m^{4} -1$ & $m^{4}-1-4m^{2}$  & $m^{4}-1+4m^{2}$ & $r(E_{m})$ & $s_{2}(E_{m})$ \\
        \hline
        $6$ & $5 \cdot 7 \cdot 37$ & $1439$ & $1151$ & $2$ & $4$ \\
        \hline
        $12$ & $13 \cdot 11 \cdot 5 \cdot 29$ & $19 \cdot 1061$ & $101 \cdot 211$ & $3$ & $3$ \\
        \hline
        $30$ & $17 \cdot 29 \cdot 31 \cdot 53$ & $11 \cdot 73309$ & $19 \cdot 42821$ & $3$ & $3$ \\
        \hline
        $42$ & $41 \cdot 43 \cdot 1765$ & $59 \cdot 101 \cdot 521$ & $229 \cdot 13619$ & $\geq 2$ & $4$ \\
        \hline
        $60$ & $13 \cdot 59 \cdot 61 \cdot 277$ & $229 \cdot 56531$  & $31 \cdot 139 \cdot 3011$ & $4$ & $4$ \\
        \hline
        $462$ & $5 \cdot 461 \cdot 463 \cdot 42689$ & $45557487359$ & $45559194911$ & $\geq 3$ & $\geq 5$ \\
        \hline
    \end{tabular}
    \caption{Explicit values/lower bounds of $s_{2}(E_{m})$ and $r(E_{m})$}
    \label{tab:my_label}
\end{table}

In what follows, we adhere to the following notations throughout the paper. 
\subsection*{Notation} 
\begin{itemize}
\item For any $\alpha \in \mathbb{Q}^{*}$, its canonical image in $\mathbb{Q}^{*}/(\mathbb{Q}^{*})^{2}$ is denoted by $[\alpha]$. Also, for $[\alpha_{1}],\ldots,[\alpha_{t}] \in \mathbb{Q}^{*}/(\mathbb{Q})^{2}$, the subgroup of $\mathbb{Q}^{*}/(\mathbb{Q}^{*})^{2}$ generated by these elements is denoted by $\langle [\alpha_{1}],\ldots, [\alpha_{t}] \rangle$. 
 \item Here $m$ denotes an even positive integer such that both $m + 1$ and $m - 1$ are prime numbers and $m^{2}+1$ is square-free.
    \item By $p_{i}$ (resp. $q_{i}$ and $r_{i}$), we denote all the prime factors of $m^{4}-1$ (resp. all prime factors of $m^{4}-1-4m^{2}$ and $m^{4} - 1 + 4m^{2}$).
    \item An arbitrary place is denoted by $\ell \leq \infty$. Similarly, $z \in \mathbb{Q}_{\ell}$ is written as $z = u \cdot \ell^{t}$, where $u \in \mathbb{Z}_{\ell}^{*}$. In that case, the $\ell$-adic valuation of $z$ is $v_{\ell}(z) = t$. 
    \item For a prime number $\ell$, the symbol $\left(\frac{a}{\ell}\right)$ stands for the Legendre symbol of $a \pmod \ell$.
    %\item For a positive integer $n$, $\Omega_{n,i}$ denotes the number of prime factors of $n$ that are $i \pmod 8$.
    \item By $r(E)$ and $s_{2}(E)$, we denote respectively the Mordell-Weil rank and the $2$-Selmer rank of an elliptic curve $E$.
    \item For an elliptic curve $E$ and a point $(x,y) \in E(\mathbb{Q})$, we denote its canonical image in $E(\mathbb{Q})/2E(\mathbb{Q})$ by $(\bar{x},\bar{y})$. 
\end{itemize}
 
 The main result of this paper is as follows. 

\begin{theorem}\label{mainthm}
    Let $E_{m}: y^{2} = x (x - (m^{2}+1)^{2})(x + (m^{2}-1)^{2}) + (2m(m^{4}-1))^{2}$ where $m$ is an even integer such that $m \pm 1$ are primes, and $m^{2}+1$ is square-free. Then $E_{m}(\mathbb{Q})_{\rm{tors}} \simeq \mathbb{Z}/ 2 \mathbb{Z} \times \mathbb{Z}/ 2 \mathbb{Z}$ and $r(E_{m}) \geq 2$.\\
    Moreover, $s_{2}(E_{m}) \geq w$, where $w$ denotes the number of prime factors $p$ of $m^{4}-1$ satisfying $\left(\frac{p}{q_{i}}\right) = \left(\frac{p}{r_{j}}\right) = 1$ if $p \equiv 1 \pmod 4$, and $\left(\frac{p}{q_{i}}\right) = \left(\frac{-p}{r_{i}}\right) = 1$ if $p \equiv 3 \pmod 4$, where $q_{i}$ (resp. $r_{j}$) varies over all the prime factors of $m^{4}-1-4m^{2}$ (resp. over all the prime factors of $m^{4}-1+4m^{2}$).   
\end{theorem}

\begin{remark}
    It is a folklore conjecture that almost all the elliptic curves over $\mathbb{Q}$ have Mordell-Weil rank either $0$ or $1$ and therefore, those with bigger rank are rare to find. Our result is significant in the sense that it provides a family of elliptic curves of rank at least $2$ . 
\end{remark}

\begin{remark}
    We note that, if $\ell$ is a prime factor of $m^{2} + 1$, then $-1 \equiv m^{2} \pmod {\ell}$ implies that $\left(\frac{-1}{\ell}\right) = \left(\frac{m^{2}}{\ell}\right) = 1$. Therefore, we have $\ell \equiv 1 \pmod {4}$. Again, since both $m - 1$and $m + 1$ are prime numbers, we conclude that exactly one of them is of the form $4k + 3$. Hence only one of all the prime factors of $m^{4}-1$ is of the form $4k+3$.  
\end{remark}

\begin{cor}\label{maincor}
Let $E_{m}: y^{2} = x (x - (m^{2}+1)^{2})(x + (m^{2}-1)^{2}) + (2m(m^{4}-1))^{2}$ where $m$ is an even integer such that $m \pm 1$ are primes, and $m^{2}+1$ is squarefree. Moreover, if both $(m^{4}-1-4m^{2})$ and $(m^{4}-1+4m^{2})$ are prime numbers, the $2$-Selmer rank $s_{2}(E_{m}) = w+1$, where $w$ denotes the number of prime factors of $m^{4}-1$.
\end{cor}

\begin{remark}
In Table 1, we see that for $m = 462$, the hypotheses of Corollary \ref{maincor} are satisfied. Therefore, $s_{2}(E_{m}) = 5$. If the parity conjecture, which asserts that the Mordell-Weil rank of an elliptic curve over $\mathbb{Q}$ has the same parity as the $2$-Selmer rank, holds true then we conclude that $r(E_{m})$ is odd and therefore is at least $3$. 
\end{remark}

%\begin{remark}\label{remark}
   % One can immediately note from Corollary \ref{maincor} that if both $(m^{4}-1-4m^{2})$ and $(m^{4}-1-4m^{2})$ are prime numbers, the Mordell-Weil rank $r(E_{m}) \geq 3$ if the number of prime factors of $m^{2}+1$ is even. Now, $r(E_{m}) \geq 2$ always and $s_{2}(E_{m})$ is odd in this case from Corollary \ref{maincor}. Assuming the Shafarevich-Tate group $\Sha(E_{m})$ is finite, and noting that $r(E_{m}) \leq s_{2}(E_{m})$, we can conclude $r(E_{m}) \geq 3$.
%\end{remark}

%\begin{remark}
%    The statement of Theorem \ref{mainthm} and Remark \ref{remark} assume both $(m^{4}-1-4m^{2})$ and $(m^{4}-1+4m^{2})$ are prime numbers in the context of the elliptic curve $E_{m}$. But, the $2$-descent computation for $E_{m}$ in Section 4 for $\text{Sel}_{2}(E)$ is done in a somewhat more general setup without those assumptions until the very end.
%\end{remark}

\section{The torsion group $E_{m}(\mathbb{Q})_{\rm{tors}}$}\label{tor-sec}
As mentioned in Remark \ref{mainrmk}, we note that $E_{m}$ can be described as $E_{m}: y^{2} = (x-m_{1})(x-m_{2})(x-m_{3})$, where $m_{1} = m^{4}-1, m_{2} = 1 - m^{4}$, and $m_{3} = 4m^{2}$. It is a basic fact that a rational point $(x,y)$ on any elliptic curve is of order $2$ if and only if $y = 0$. Thus we can immediately notice that $$\{\mathcal{O}, (m_{1},0), (m_{2}, 0), (m_{3}, 0)\} \subseteq E_{m}[2],$$ where $E_{m}[2]$ is the $2$-torsion subgroup of $E_{m}(\mathbb{Q})_{\rm{tors}}$. 

\smallskip

We know that for any prime $\ell$ of good reduction, $E_{m}(\mathbb{Q})_{\rm{tors}}$ injects into $E_{m}(\mathbb{F}_{\ell})$. By our hypotheses, since both $m - 1$ and $m + 1$ are prime numbers, we have $m \equiv 0 \pmod {3}$. Thus $3$ does not divide $2^{6}\cdot (m^{4} - 1)^{2}\cdot (m^{4} - 1 -4m^{2})^{2}\cdot (m^{4} - 1 + 4m^{2})^{2}$ and therefore $3$ is a prime of good reduction. Now, we claim that $|E_{m}(\mathbb{F}_{3})| = 4$. The elements of $\mathbb{F}_{3}\times \mathbb{F}_{3}$ are $(\bar{0},\bar{0}), (\bar{0},\bar{1}), (\bar{0},\bar{2}), (\bar{1},\bar{0}), (\bar{1},\bar{1}), (\bar{1},\bar{2})$, $(\bar{2},\bar{0}), (\bar{2},\bar{1})$ and $(\bar{2},\bar{2})$. Since $3 \mid m$, we obtain that the reduced curve over $\mathbb{F}_{3}$ is $y^{2} = x^{3} - x$ that has only three solutions over $\mathbb{F}_{3}$ which are $(\bar{0},\bar{0}), (\bar{1},\bar{0})$ and $(\bar{2},\bar{0})$. Therefore, $E_{m}(\mathbb{F}_{3}) = \{\mathcal{O},(\bar{0},\bar{0}), (\bar{1},\bar{0}),(\bar{2},\bar{0})\}$. Since $E_{m}[2] \subseteq E_{m}(\mathbb{Q})_{\rm{tors}}$, we conclude that $$E_{m}(\mathbb{Q})_{\rm{tors}} = \{\mathcal{O}, (m_{1},0), (m_{2},0), (m_{3},0)\} \simeq \mathbb{Z}/2\mathbb{Z} \times \mathbb{Z}/2\mathbb{Z}.$$ $\hfill\Box$

\section{Lower bound of $r(E_{m})$ via canonical height computation}\label{lower-sec}

We prove that $r(E_{m}) \geq 2$ by establishing that the points $P_{1} = (0,t)$ and $P_{2} = (n_{1},t)$ are $\mathbb{Z}$-linearly independent. We achieve this by using the canonical height of $P_{1}$ and $P_{2}$ and showing that the height matrix thus formed is non-singular. Our treatment in this section closely follows that given in \cite{integers}.  

\smallskip

We recall that for $\alpha = \frac{p}{q}$, the height of $\alpha$ is defined by $h(\alpha) = \log(\max\{|p|,|q|\})$. Using this, we can define a height function $H$ on the set of all rational points of $E_{m}$ by declaring $H(P) = h(x(P))$, where $P = (x(P),y(P)) \in E_{m}(\mathbb{Q})$. Finally, we define the canonical height (also known as the N\'{e}ron-Tate height) $\hat{h}$ of $P$ by setting $\hat{h}(P) = \frac{1}{2} \displaystyle\lim_{N \to \infty}\frac{H(2^{N}P)}{4^{N}}$ whence it is well-known that the limit always exists.

\smallskip

The canonical height is quite useful in proving that certain rational points on an elliptic curve are $\mathbb{Z}$-linearly independent. In the following, we record a few results that will be useful in proving that $r(E_{m}) \geq 2$. 

\begin{proposition} \cite[Proposition 8]{integers}\label{integer-proposition}
For an elliptic curve $E$ defined over $\mathbb{Q}$ together with the canonical height $\hat{h}$, the following results hold true.
\begin{enumerate}
\item For $R \in E(\mathbb{Q})$, we have $\hat{h}(R) \geq 0$ and equality holds if and only if $R \in E(\mathbb{Q})_{{\rm{tors}}}$. 

\medskip

\item For $R_{1}, R_{2} \in E(Q)$, we have $\hat{h}(R_{1} + R_{2}) + \hat{h}(R_{1} - R_{2}) = 2\hat{h}(R_{1}) + 2\hat{h}(R_{2})$. This is also sometimes referred to as the parallelogram law due to the analogy with the parallelogram law in euclidean geometry.

\medskip

\item For $R \in E(\mathbb{Q})$ and $s \in \mathbb{Z}$, we have $\hat{h}(sR) = s^{2}\hat{h}(R)$.
\end{enumerate}
\end{proposition}

The canonical height pairing is the map from $E(\mathbb{Q}) \times E(\mathbb{Q}) \to \mathbb{R}$ defined by the equation $\langle R_{1},R_{2} \rangle = \hat{h}(R_{1} + R_{2}) - \hat{h}(R_{1}) - \hat{h}(R_{2})$. It is known that this is a symmetric bilinear map. Moreover, if $R_{1},\ldots,R_{k} \in E(\mathbb{Q})$, then they are $\mathbb{Z}$-linearly independent if and only if the height-pairing $k \times k$-matrix $[\langle R_{i},R_{j}\rangle]_{1 \leq i,j \leq k}$ is invertible.  

\smallskip

Now, we compute the determinant of the height-pairing $2 \times 2$-matrix for the points $P_{1}$ and $P_{2}$ on $E_{m}$.  For $P_{1} = (0,t)$, using the duplication formula \cite[Algorithm 2.3.]{Sil}, we get $x(2P_{1}) = \frac{m^{8} + 62m^{4} + 1}{16m^{2}}$. We observe that the numerator is a polynomial of $m$ of degree $8$. Inductively, we see that for any integer $N \geq 1$, the numerator of the $x$ co-ordinate of $2^{N}P_{1}$ is a polynomial of degree $2\cdot 4^{N}$. Also, the denominator is a polynomial of $m$ of degree strictly smaller than that of the numerator. Therefore, by considering the logarithm to the base $m$ we see that if $x(2^{N}P_{1}) = \frac{p_{N}}{q_{N}}$, then $$H(2^{N}P_{1}) = \log(\max\{|p_{N}|,|q_{N}|\} ) = \log(|p_{N}|) = \log(m^{2\cdot 4^{N}} + h_{N}(m)),$$ where $h_{N}(m)$ is a polynomial in $m$ of degree strictly smaller than $2\cdot 4^{N}$. Consequently, the canonical height of $P_{1}$ turns out to be $$\hat{h}(P_{1}) = \frac{1}{2}\displaystyle\lim_{N \to \infty}\frac{H(2^{N}P_{1})}{4^{N}} = \frac{1}{2}\displaystyle\lim_{N \to \infty}\frac{2\cdot 4^{N}}{4^{N}} = 1.$$ Using similar technique, we find that $\hat{h}(P_{2}) = \frac{3}{2}$. Using Proposition \ref{integer-proposition}, we find that $$\langle P_{1},P_{1}\rangle = \hat{h}(P_{1} + P_{1}) - \hat{h}(P_{1}) - \hat{h}(P_{1}) = 4\hat{h}(P_{1}) - 2\hat{h}(P_{1}) = 2.$$ By repeatedly making use of Proposition \ref{integer-proposition}, we finally obtain the height-pairing matrix as $\begin{bmatrix}
    2 & -1 \\
    -1 & 3 
\end{bmatrix}$ which is non-singular. Therefore, the points $P_{1}$ and $P_{2}$ are independent and thus $r(E_{m}) \geq 2$. $\hfill\Box$
\begin{remark}
It is easy to see that the points $P_{1} = (0,t)$, $P_{2} = (n_{1},t)$ and $P_{3} = (n_{2},t)$ are linearly dependent because they are collinear in the Euclidean plane. Hence $P_{1} + P_{2} + P_{3} = \mathcal{O}$ is a linear relation among them. One can also compute the height-pairing matrix for these points and this turns out to be $\begin{bmatrix}
2 & -1 & -1 \\
-1 & 3 & -2 \\
-1 & -2 & 3
\end{bmatrix}$ which is singular as expected. But this $3 \times 3$ matrix also illustrates that $P_{1},P_{3}$ are independent and $P_{2},P_{3}$ are also independent. 
\end{remark}

\section{An investigation of the 2-Selmer rank ${\rm{Sel}}_{2}(E_{m})$}

\noindent The $2$-Selmer group of the elliptic curve $E_{m} : y^{2} = (x - (m^{4}-1)) (x + (m^{4}-1)) (x - 4m^{2})$ over $\mathbb{Q}$ is denoted by ${\rm{Sel}}_{2}(E_{m})$ and the $2$-Selmer rank of $E_{m}$, denoted by $s_{2}(E_{m})$, is defined by $|{\rm{Sel}}_{2}(E_{m})| = 2^{2+s_{2}(E_{m})}$. As noted earlier, the discriminant of $E_{m}$ is a divisor of $64\cdot (m^{4}-1)^{2}\cdot (m^{4}-1-4m^{2})^{2} \cdot (m^{4}-1+4m^{2})^{2}$. Let $S$ denote the set consisting of all finite places at which $E_{m}$ has bad reductions, the infinite places, and the rational prime $2$. We define
\begin{align}
\label{pairs}
\mathbb{Q}(S,2)&=\left\{[b]\in\mathbb{Q}^*/(\mathbb{Q}^*)^2 :
\text{ord}_\ell(b)\equiv 0 ~(\bmod \text{ } {2}) ~ \text{for all primes} ~ \ell\not \in S\right\}\\
&=\langle{[\pm 2], \;[\pm p_i], [\pm q_i], \cdots [\pm r_i] \; \rangle}\nonumber
\end{align}
where $p_{i}$ ranges over distinct prime factors of $(m^{4}-1)$, $q_{i}$ and $r_{i}$ ranges over distinct prime factors of $(m^{4}-1-4m^{2})$, and $(m^{4}-1+4m^{2})$ respectively, that are bad primes. By the method of 2-descent (see \cite{Sil} Proposition X.1.4), there exists an injective group homomorphism
$$ \phi: E_{m}(\mathbb{Q})/2E_{m}({\mathbb{Q}}) \longrightarrow \mathbb{Q}(S,2) \times \mathbb{Q}(S,2)$$
 defined by 
\begin{align*}
\phi(\bar{x},\bar{y}) = \begin{cases}
([x - (m^{4}-1)], [x + (m^{4}-1)])  & \text{if} \text{  } x\neq \pm (m^{4}-1), \\
([2(m^{4}-1)(m^{4}-1-4m^{2})], [2(m^{4}-1)]) & \text{if} \text{  } x = m^{4}-1,\\
([-2(m^{4}-1)], [2(m^{4}-1)(m^{4}-1+4m^{2})]) & \text{if}\text{  } x = -(m^{4}-1), \\
([1], [1]) &\text{if} \text{  } (x,y) = \mathcal{O},
\end{cases}
\end{align*}
where $\mathcal{O}$ is the fixed base point. If $([b_1], [b_2])$ is a pair which is not in the image of the cosets of the torsion points of $E_{m}(\mathbb{Q})$, then $([b_1], [b_2])$ is the image of a point $ P = (\bar{x}, \bar{y}) \in E_{m}(\mathbb{Q})/2E_{m}(\mathbb{Q})$ if and only if the following equations
\begin{align}
& b_1z_1^2 - b_2z_2^2 = -2 \cdot (m^{4}-1), \label{eq1}\\
& b_1z_1^2 - b_1b_2z_3^2 = -(m^{4}-1-4m^{2}), \label{eq2}\\
& b_1b_2z_3^2 - b_2z_2^2 = - (m^{4}-1 + 4m^{2}) \label{eq3}
\end{align}
have a solution $(z_1,z_2,z_3) \in \mathbb{Q}^* \times \mathbb{Q}^* \times \mathbb{Q}^{*}$. We note that (\ref{eq3}) is obtained by subtracting (\ref{eq2}) from (\ref{eq1}), and is only included here due to its use later in this work. The image of $E_{m}(\mathbb{Q})/2E_{m}({\mathbb{Q}})$ under the $2$-descent map $\phi$ is contained in a subgroup of $\mathbb{Q}(S,2)\times \mathbb{Q}(S,2)$ known as the \textit{$2$-Selmer group} ${\rm{Sel}}_2(E_{m}/\mathbb{Q})$, which fits into the  exact sequence

\begin{equation}
\label{selmer}
0 \longrightarrow E_{m}(\mathbb{Q})/2E_{m}(\mathbb{Q}) \longrightarrow {\rm{Sel}}_2 (E_{m}/\mathbb{Q}) \longrightarrow \Sha(E_{m}/\mathbb{Q})[2]\longrightarrow 0.
\end{equation}
\begin{remark}\label{rems1}
We note that the choice of $([b_{1}],[ b_{2}]) \in {\rm{Sel}}_{2}(E)$ for any $([b_{1}], [b_{2}]) \in \mathbb{Q}(S,2) \times \mathbb{Q}(S,2)$ is the same as the four elements in the equivalence class represented by the image of $([b_{1}], [b_{2}])$ in $\mathbb{Q}(S,2) \times \mathbb{Q}(S,2) / \phi(E_{m}(\mathbb{Q})_{\rm{tors}})$. Hence, without loss of generality, we can assume $b_{1}b_{2} \not \equiv 0 \pmod 4$ adding to the conditions mentioned in Lemma \ref{lems1} while looking for possible $([b_{1}], [b_{2}]) \in {\rm{Sel}}_{2}(E_{m})$. This is because $\phi(E(\mathbb{Q})_{\rm{tors}})$ contains $(b_{1}, b_{2})$ such that both $b_{1}, b_{2}$ are even square-free integers. 
\end{remark}
%\noindent As mentioned earlier, throughout this section, $p_{i}, q_{i}, r_{i}$ denote all possible prime factors of $(m^{4}-1), (m^{4}-1 -4m^{2})$, and $(m^{4}-1+4m^{2})$ respectively.\\
We start with local solutions to the homogeneous spaces defined by \eqref{eq1} and \eqref{eq2}. For a prime number $\ell$, we denote an $\ell$-adic solution for \eqref{eq1} and \eqref{eq2} as $z_{i} = u_{i} \cdot \ell^{t_{i}}$ for $i = 1,2,3$ where $u_{i} \in \mathbb{Z}_{\ell}^{*}$. We note that this implies that the $\ell$-adic valuation $v_{\ell}(z_{i}) = t_{i}$. 
\begin{lemma}\label{lems1}
    Suppose equation \eqref{eq1} and equation \eqref{eq2} have a solution $(z_1,z_2,z_3) \in \mathbb{Q}_{\ell} \times \mathbb{Q}_{\ell} \times \mathbb{Q}_{\ell}$ for some odd prime number $\ell$. Then
    \begin{enumerate}[(i)]
    \item $v_{\ell}(z_{1}) = -\lambda < 0$ if and only if $v_\ell(z_2)=-\lambda < 0$ for some integer $\lambda$.
    \item $v_{\ell}(z_{1}) = -\lambda < 0$ implies $v_\ell(z_3)=-\lambda < 0$ for some integer $\lambda$.
    \item $v_{\ell}(z_{3}) = -\lambda < 0$ implies $v_\ell(z_1)=-\lambda < 0$ for some integer $\lambda$ unless $b_{1}b_{2} \equiv 0 \pmod {\ell^{2}}$ and $m^{4}-1 \equiv 0 \pmod \ell$. That is, $\ell$ varies over the set of primes $p_{i}$'s, in which case, $v_{\ell}(z_{3}) = -1, v_{\ell}(z_{1}) \geq 0$ is a possibility. 
    \end{enumerate}
\end{lemma}
\begin{proof}
To Prove $(i)$, first we assume that $v_{\ell}(z_{1}) = -\lambda < 0$ and let $v_{\ell}(z_{2}) = t_{2}$. Then equation \eqref{eq1} becomes $$b_{1}(u_{1}\ell^{-\lambda})^{2} - b_{2}(u_{2}\ell^{t_{2}})^{2} = -2(m^{4} - 1).$$ This implies $b_{1}u_{1}^{2} - b_{2}u_{2}^{2}\ell^{2(\lambda + t_{2})} = 2\ell^{2\lambda}(m^{4} - 1)$. If $\lambda + t_{2} > 0$, then we have $b_{1}u_{1}^{2} \equiv 0 \pmod {\ell^{2}}$ and since $u_{1}$ is a unit in $\mathbb{Z}_{\ell}$, we conclude that $b_{1} \equiv 0 \pmod {\ell^{2}}$ which is a contradiction to the fact that $b_{1}$ is square-free. Again, if $\lambda + t_{2} < 0$, then equation \eqref{eq1} becomes $$b_{1}u_{1}^{2}\ell^{-2(\lambda + t_{2})} - b_{2}u_{2}^{2} = -2(m^{4} - 1)\ell^{2\lambda - 2(\lambda + t_{2})},$$ which implies that $b_{2}u_{2}^{2} \equiv 0 \pmod {\ell^{2}}$. Now, $u_{2} \in \mathbb{Z}_{\ell}^{*}$ implies that $b_{2} \equiv 0 \pmod {\ell^{2}}$, a contradiction to the fact that $b_{2}$ is square-free. Thus we conclude that $v_{\ell}(z_{2}) = -\lambda$. Using a similar argument, we can also establish that if $v_{\ell}(z_{2}) = -\lambda < 0$, then $v_{\ell}(z_{1}) = -\lambda$ as well.

\smallskip

To prove $(ii)$, we assume that $v_{\ell}(z_{1}) = -\lambda < 0$ and let $v_{\ell}(z_{3}) = t_{3}$. Then equation \eqref{eq2} translates into $$b_{1}u_{1}^{2} - b_{1}b_{2}u_{3}^{2}\ell^{2(\lambda + t_{3})} = -\ell^{2\lambda}(m^{4} - 1 - 4m^{2}).$$ If $\lambda + t_{3} > 0$, then we have $b_{1}u_{1}^{2} \equiv 0 \pmod {\ell^{2}}$ and consequently, $b_{1} \equiv 0 \pmod {\ell^{2}}$, a contradiction to the fact that $b_{1}$ is square-free. Thus $\lambda + t_{3} < 0$ and by writing $t_{3} = -\lambda - k$ for some $k > 0$, we obtain from equation \eqref{eq2} that $b_{1}b_{2} \equiv 0 \pmod {\ell^{2k}}$. Since both $b_{1}$ and $b_{2}$ are square-free, we conclude that $k \leq 1$. 

\smallskip

If $k = 1$, then we have $\ell^{2} \mid b_{1}b_{2}$ and since both $b_{1}$ and $b_{2}$ are square-free, we have $\ell \mid b_{1}$ and $\ell \mid b_{2}$. Then from equation \eqref{eq2}, we obtain $$b_{1}u_{1}^{2} - \frac{b_{1}b_{2}}{\ell^{2}}u_{3}^{2} = -\ell^{2\lambda}(m^{4} - 1 - 4m^{2}).$$ Since $\ell$ divides both $b_{1}$ and the right hand side of the above equation, we conclude that $\frac{b_{1}b_{2}}{\ell^{2}} \equiv 0 \pmod {\ell}$. That is, $\ell^{3} \mid b_{1}b_{2}$ which contradicts the fact that both $b_{1}$ and $b_{2}$ are square-free. Hence $k = 1$ is impossible to hold and therefore, $k = 0$. In other words, $v_{\ell}(z_{3}) = -\lambda$. 

\smallskip

To prove $(iii)$, assume that $v_{\ell}(z_{3}) = -t_{3} < 0$. That is, $z_{3} = u_{3} \cdot \ell^{-t_{3}}$ where $u_{3} \in {\mathbb{Z}_{\ell}}^{*}$. If $v_{\ell}(z_{1}) = t_{1} < -t_{3}$, then from \eqref{eq2}, we get $$b_{1}u_{1}^{2} - b_{1}b_{2}u_{3}^{2}\ell^{2(-t_{1}-t_{3})} = -(m^{4} - 1 - 4m^{2}) \cdot \ell^{2t_{1}} \implies b_{1} \equiv 0 \pmod {\ell^{2}}, \text{ a contradiction}.$$

\smallskip
 
Now suppose $v_{\ell}(z_{1}) = t_{1} > -t_{3}$. Then again from \eqref{eq2}, we get $$b_{1}u_{1}^{2}\ell^{2(t_{1} + t_{3})} - b_{1}b_{2}u_{3}^{2} = -(m^{4} - 1 - 4m^{2}) \cdot \ell^{2t_{3}} \implies b_{1}b_{2} \equiv 0 \pmod {\ell^{2}}.$$

\smallskip

Hence, $b_{1}b_{2} \not \equiv 0 \pmod {\ell^{2}}$ implies $v_{\ell}(z_{1}) > v_{\ell}(z_{3})$ is also not possible. That is, $v_{\ell}(z_{3}) = v_{\ell}(z_{1})$. Also, $b_{1}b_{2} \equiv 0 \pmod {\ell^{2}} $ implies $\ell$ divides both $b_{1}, b_{2}$. This in turn implies that $t_{3} \geq 2$ or $(m^{4} - 1 - 4m^{2}) \equiv 0 \pmod \ell$ are not possible, as both cases then imply $b_{1}b_{2} \equiv 0 \pmod {\ell^{3}}$, a contradiction. Equation \eqref{eq3} under the assumption $v_{\ell}(z_{1}) > v_{\ell}(z_{3})$ now looks like $$\frac{b_{1}b_{2}}{\ell^{2}}u_{3}^{2} - b_{2}z_{2}^{2} = -(m^{4}-1+4m^{2}) \implies (m^{4}-1+4m^{2}) \not \equiv 0 \pmod \ell.$$
    This is because $b_{2} \equiv 0 \pmod \ell$ and $v_{\ell}(z_{2}) < 0$ is not possible from part $(1)$ of this result as $v_{\ell}(z_{1}) > v_{\ell}(z_{3}) = -1$. Noting that $b_{1}, b_{2}$ are square-free combinations of the prime factors of $(m^{4}-1), (m^{4}-1-4m^{2})$ and $(m^{4}-1+4m^{2})$ and $\ell$ divides both $b_{1}$ and $b_{2}$ in this case, we conclude the proof. 
\end{proof}
The following lemma reduces the potential size of ${\rm{Sel}}_{2}(E_{m})$ by excluding certain choices of $b_{1}$ and $b_{2}$ due to the lack of local solution of \eqref{eq1} and \eqref{eq2} for each prime $\ell \leq \infty$. 
\begin{lemma}\label{lems2}
    Let $([b_{1}], [b_{2}]) \in \mathbb{Q}(S,2) \times \mathbb{Q}(S,2)$. Then $([b_{1}], [b_{2}]) \not \in {\rm{Sel}}_{2}(E_{m})$ if
    \begin{enumerate}[(i)]
    \item $b_{2} < 0$ due to no solution of \eqref{eq1} and \eqref{eq3} over $\mathbb{Q}_{\infty}$.
    \item $b_{2} \equiv 0 \pmod {q_{i}}$ due to no solution of \eqref{eq1} and \eqref{eq2} over $\mathbb{Q}_{q_{i}}$.
    \item $b_{1} \equiv 0 \pmod {r_{i}}$ due to no solution of \eqref{eq1} and \eqref{eq2} over $\mathbb{Q}_{r_{i}}$.
    \item $v_{p_{i}}(b_{1}b_{2}) = 1$ due to no solution of \eqref{eq1} and \eqref{eq2} over $\mathbb{Q}_{p_{i}}$.
    \item $b_{1}b_{2} \equiv 2 \pmod 4$ due to no solution of \eqref{eq1} and \eqref{eq2} over $\mathbb{Q}_{2}$.
    \end{enumerate}
\end{lemma}
\begin{proof}
    \begin{enumerate}[(i)]
    \item For $b_{2} < 0$, we note that if $b_{1} > 0$, then $-2(m^{4}-1) > 0$ from \eqref{eq1}, a contradiction. Similarly, for $b_{1} < 0$, we get $-(m^{4}-1+4m^{2}) > 0$ from \eqref{eq3}, again a contradiction. Thus $b_{2} < 0$ implies that $([b_{1}],[b_{2}]) \not\in {\rm{Sel}}_{2}(E_{m})$. \\
    \item Let us assume $b_{2} \equiv 0 \pmod {q_{i}}$. If $v_{q_{i}}(z_{j}) = -t < 0$ for all $j \in \{1,2,3\}$, then from \eqref{eq1} we get $$b_{1}u_{1}^{2} - b_{2}u_{2}^{2} = - 2 (m^{4}-1)q_{i}^{2t} \implies b_{1} \equiv 0 \pmod {q_{i}}.$$ But then from \eqref{eq2}, one can get the following contradiction. $$b_{1}u_{1}^{2} - b_{1}b_{2}u_{3}^{2} = -(m^{4} - 1 - 4m^{2})q_{i}^{2t} \implies b_{1} \equiv 0 \pmod {q_{i}^{2}}.$$
    Now if $v_{q_{i}}(z_{j}) \geq 0$ for all $j \in \{1,2,3\}$, equation \eqref{eq2} again implies that either $b_{1} \equiv 0 \pmod {q_{i}}$ or $v_{q_{i}}(z_{1}) > 0$ holds. Either of which then implies $2(m^{4}-1) \equiv 0 \pmod {q_{i}}$ from \eqref{eq1}, a contradiction as $\gcd(2(m^{4}-1), (m^{4}-1-4m^{2})) = 1$. Therefore, $b_{2} \equiv 0 \pmod {q_{i}}$ implies that $([b_{1}],[b_{2}]) \not\in {\rm{Sel}}_{2}(E_{m})$. \\
    \item Assume that $b_{1} \equiv 0 \pmod {r_{i}}$. If $v_{r_{i}}(z_{j}) = -t < 0$ for all $j \in \{1,2,3\}$, then from equation \eqref{eq1}, we obtain $$b_{1}u_{1}^{2} - b_{2}u_{2}^{2} = -2r_{i}^{2t}(m^{4} - 1),$$ which implies that $b_{2} \equiv 0 \pmod {r_{i}}$. Again, from equation \eqref{eq3}, we obtain $b_{1}b_{2}u_{3}^{2} - b_{2}u_{2}^{2} = -r_{i}^{2t}(m^{4} - 1 + 4m^{2})$. From this and $b_{1} \equiv b_{2} \equiv 0 \pmod {r_{i}}$, we have $b_{2} \equiv 0 \pmod {r_{i}^{2}}$, a contradiction. 
    
    \smallskip
    
Again, if $v_{r_{i}}(z_{j}) = t \geq 0$ for all $j \in \{1,2,3\}$, then from equation \eqref{eq3}, we obtain $$b_{1}b_{2}(u_{3}r_{i}^{t})^{2} - b_{2}(u_{2}r_{i}^{t})^{2} = -(m^{4} - 1 + 4m^{2}).$$ Now, $r_{i}$ divides the right hand side of the above equation which is square-free under our hypotheses. Hence the $r_{i}$-adic valuation of the left hand side of the equation is $1$. Since $b_{1} \equiv 0 \pmod {r_{i}}$, we must have either $b_{2} \equiv 0 \pmod {r_{i}}$ or $t > 0$. Now using equation \eqref{eq1}, we see that that in either case we have $b_{1}z_{1}^{2} - b_{2}z_{2}^{2} \equiv 0 \pmod {r_{i}}$. That is, $r_{i} \mid 2(m^{4} - 1)$, a contradiction. Consequently, $b_{1} \equiv 0 \pmod {r_{i}}$ implies that $([b_{1}],[b_{2}]) \not\in {\rm{Sel}}_{2}(E_{m})$. \\
\item Assume that $v_{p_{i}}(b_{1}b_{2}) = 1$. Since $b_{1}$ and $b_{2}$ are square-free, we have that $p_{i}$ divides exactly one of $b_{1}$ and $b_{2}$. We deal only with the case $p_{i} \mid b_{1}$ as the other case follows exactly a similar line of argument. 

\smallskip

Now, of $v_{p_{i}}(z_{j}) = -t < 0$ for all $j \in \{1,2,3\}$, then equation \eqref{eq1} yields $$b_{1}u_{1}^{2} - b_{2}u_{2}^{2} = -2p_{i}^{2t}(m^{4} - 1),$$ which implies that $p_{i} \mid b_{2}$, a contradiction. Again, if $v_{p_{i}}(z_{j}) = t > 0$ for all $j \in \{1,2,3\}$, then equation \eqref{eq2} yields $p_{i} \mid (m^{4} - 1 + 4m^{2})$, a contradiction. Consequently, $v_{p_{i}}(b_{1}b_{2}) = 1$ implies that $([b_{1}],[b_{2}]) \not\in {\rm{Sel}}_{2}(E_{m})$. \\
\item Without loss of any generality, let us assume that $b_{1}$ is odd and $b_{2}$ is even. Then $v_{2}(z_{j}) = -t < 0$ for all $j \in \{1,2,3\}$ implies $b_{1} \equiv 0 \pmod 2$ from equation \eqref{eq1}, a contradiction. Also, if $v_{2}(z_{j}) \geq 0$, then from equation \eqref{eq3}, we get $$b_{1}b_{2}z_{3}^{2} - b_{2}z_{2}^{2} = -(m^{4} - 1 + 4m^{2}) \equiv 0 \pmod 2,$$ a contradiction as $m$ is an even integer. Hence $b_{1}b_{2} \equiv 2 \pmod {4}$ implies that $([b_{1}],[b_{2}]) \not\in {\rm{Sel}}_{2}(E_{m})$. This completes the proof of Lemma \ref{lems2}. 
    \end{enumerate}
\end{proof}
%\noindent In the following result, we now introduce the Legendre symbol $(\frac{.}{.})$ to help determine the size of $\text{Sel}_{2}(E)$.
\begin{lemma}\label{lems3}
    Let $([b_{1}], [b_{2}]) \in \mathbb{Q}(S,2) \times \mathbb{Q}(S,2)$ be such that $([b_{1}], [b_{2}]) \in \text{Sel}_{2}(E_{m})$. Then 
    \begin{enumerate}[(i)]
    \item if $b_{1}b_{2} \equiv 0 \pmod {p_{i}^{2}}$ then $\Big(\frac{-b_{1}b_{2}/p_{i}^{2}}{p_{i}}\Big) = 1$. Otherwise, $\left(\frac{b_{1}}{p_{i}}\right) = \left(\frac{b_{2}}{p_{i}}\right)$.
    \item if $b_{1} \equiv 0 \pmod {q_{i}}$, then $\Big(\frac{b_{2}}{q_{i}}\Big) = \Big(\frac{2}{q_{i}}\Big)$. Otherwise, $\Big(\frac{b_{2}}{q_{i}}\Big) = 1$.
    \item if $b_{2} \equiv 0 \pmod {r_{i}}$, then $\Big(\frac{b_{1}}{r_{i}}\Big) = \Big(\frac{2}{r_{i}}\Big)$. Otherwise, $\Big(\frac{b_{1}}{r_{i}}\Big) = 1$.
    \item the congruence relation $b_{1} \equiv 1 \pmod 4$ holds.
    \end{enumerate} 
    %$(v)$ $b_{1}b_{2} \equiv 0 \pmod 9$ implies $b_{1} \not \equiv b_{2} \pmod 9$.
\end{lemma}

\begin{proof}
%    We prove all but the third case because of the similarity it bears with the proof of the second case. \\ 
\begin{enumerate}[(i)]
\item If $b_{1}b_{2} \equiv 0 \pmod {p_{i}^{2}}$, we note that $v_{p_{i}}(z_{j}) \geq 0$ for all $j \in \{1,2,3\}$ is not possible from equation \eqref{eq2}. Hence $v_{p_{i}}(z_{j}) < 0$. Now from Lemma \ref{lems1}, we get either $v_{p_{i}}(z_{j}) = -t < 0$ for all $j \in \{1,2,3\}$ or $v_{p_{i}}(z_{3}) = -1$ and $v_{p_{i}}(z_{j}) \geq 0$ for $j \in \{1,2\}$. Now, if $v_{p_{i}}(z_{j}) = -t < 0$ for all $j \in \{1,2,3\}$, then equation \eqref{eq2} implies $b_{1} \equiv 0 \pmod {p_{i}^{2}}$, a contradiction. Therefore, the only possible case left is $v_{p_{i}}(z_{3}) = -1$ and $v_{p_{i}}(z_{j}) \geq 0$ for $j \in \{1,2\}$. Noting that $m^{4}-1 \equiv 0 \pmod {p_{i}}$, equation \eqref{eq2} yields $\frac{b_{1}b_{2}}{p_{i}^{2}}u_{3}^{2} \equiv (m^{4}-1-4m^{2}) \equiv -4m^{2} \pmod {p_{i}}$. Consequently, we have $\Big(\frac{-b_{1}b_{2}/p_{i}^{2}}{p_{i}}\Big) = 1$. 

\smallskip

Now, if $b_{1}b_{2} \not\equiv 0 \pmod {p_{i}^{2}}$, one can easily note that $v_{p_{i}}(z_{j}) < 0$ for all $j \in \{1,2,3\}$, or, $v_{p_{i}}(z_{1}) = v_{p_{i}}(z_{2}) = 0$ are the only possibilities. Then from equation \eqref{eq1}, we have $\left(\frac{b_{1}}{p_{i}}\right) = \left(\frac{b_{2}}{p_{i}}\right)$ in both the cases. \\

\item If $b_{1} \equiv 0 \pmod {q_{i}}$, we see that $v_{q_{i}}(z_{2}) >  0$ is impossible, because that forces $v_{q_{i}}(z_{1}) \geq 0$ as well and thus the left hand side of equation \eqref{eq1} is divisible by $q_{i}$. This is a contradiction because the right-hand side of equation \eqref{eq1} is not divisible by $q_{i}$. If $v_{q_{i}}(z_{2}) = 0$, then from \eqref{eq1}, we obtain $$b_{1}z_{1}^{2} - b_{2}u_{2}^{2} = -2(m^{4}-1) \implies -b_{2}u_{2}^{2} \equiv -8m^{2} \pmod {q_{i}} \implies \left(\frac{b_{2}}{q_{i}}\right) = \left(\frac{2}{q_{i}}\right).$$ Now $v_{q_{i}}(z_{1})  = -t < 0$ implies $v_{q_{i}}(z_{j}) = -t < 0$ for all $j \in \{1,2,3\}$. But then $$b_{1}u_{1}^{2} - b_{2}u_{2}^{2} = -2(m^{4}-1) \cdot q_{i}^{2t} \implies b_{2} \equiv 0 \pmod {q_{i}},$$ a contradiction from Lemma \ref{lems2}.

\smallskip

Now, let us assume $b_{1} \not \equiv 0 \pmod {q_{i}}$. Then from \eqref{eq2}, one can note that $v_{q_{i}}(z_{1}) \leq 0$, and $v_{q_{i}}(z_{3}) \leq 0$. From Lemma \ref{lems1}, this implies either $v_{q_{i}}(z_{1}) = v_{q_{i}}(z_{3}) = 0$ or $v_{q_{i}}(z_{j}) = -t$ for some $t > 0$ and $j \in \{1,2,3\}$. Both these conditions then imply $\left(\frac{b_{2}}{q_{i}}\right) = 1$. Hence, the result follows. \\

\item Assuming $b_{2} \equiv 0 \pmod {r_{i}}$, if $v_{r_{i}}(z_{1}) > 0$, then from Lemma \ref{lems1} it follows that $v_{r_{i}}(z_{2}) \geq 0$. Therefore, the left-hand side of equation \eqref{eq1} is divisible by $r_{i}$, whereas the right-hand side is not a contradiction. Consequently, we have $v_{r_{i}}(z_{1}) \leq 0$. If $v_{r_{i}}(z_{1}) = -t < 0$, then from Lemma \ref{lems1} it follows that $v_{r_{i}}(z_{j}) = -t$ for all $j \in \{1,2,3\}$. Therefore, equation \eqref{eq1} yields $$b_{1}u_{1}^{2} - b_{2}u_{2}^{2} = -2(m^{4} - 1)r_{i}^{2t}.$$ But then $b_{2} \equiv 0 \pmod {r_{i}}$ implies that $b_{1} \equiv 0 \pmod {r_{i}}$. This is a contradiction to $(3)$ of Lemma \ref{lems2}. Hence, $v_{q_{i}}(z_{1}) = 0$, and then from \eqref{eq1}, it follows that $\left(\frac{b_{1}}{r_{i}}\right) = \left(\frac{-2(m^{4}-1)}{r_{i}}\right) = \left(\frac{8m^{2}}{r_{i}}\right) = \left(\frac{2}{r_{i}}\right)$.

\smallskip

Now let us assume $b_{2} \not \equiv 0 \pmod {r_{i}}$. Then similar to the previous case, equation \eqref{eq3} and Lemma \ref{lems1} yield that either $v_{r_{i}}(z_{2}) = v_{r_{i}}(z_{3}) = 0$, or $v_{r_{i}}(z_{j}) = -t > 0$ for all $j \in \{1,2,3\}$. Either way, that implies $\left(\frac{b_{1}}{r_{i}}\right) = 1$. Hence the result follows. \\

\item Assume that $([b_{1}], [b_{2}]) \in {\rm{Sel}}_{2}(E_{m})$. If $v_{2}(z_{i}) < 0$ for all $i \in \{1,2,3\}$, then from equation \eqref{eq1}, we obtain $$b_{1}u_{1}^{2} - b_{2}u_{2}^{2} = -2(m^{4}-1) \cdot 2^{2t} \implies b_{1} - b_{2} \equiv 0 \pmod 4. $$ Putting that in equation \eqref{eq2}, we get $b_{1}(u_{1}^{2} - b_{2}u_{2}^{2}) \equiv b_{1}(1 - b_{1}) \equiv 0 \pmod 4$, from where it follows that $b_{1} \equiv b_{2} \equiv 1 \pmod 4$. 

\smallskip

If $v_{2}(z_{1}) > 0$, then the left-hand side of equation \eqref{eq1} is divisible by $4$, but the right-hand side is not. This, in turn, implies that $v_{2}(z_{1}) = v_{2}(z_{2}) = 0$ is the only possibility. But $v_{2}(z_{1}) = v_{2}(z_{2}) = 0$ implies $b_{1} - b_{2} \equiv -2(m^{4}-1) \equiv 2 \pmod 4$ which implies either $b_{1} \equiv 1 \pmod 4, b_{2} \equiv 3  \pmod 4$ or $b_{1} \equiv 3 \pmod {4}$, $b_{2} \equiv 1 \pmod {4}$. We now prove that the later one is not a possibility. If $v_{2}(z_{3}) > 0$, the from equation \eqref{eq2}, it follows that $b_{1} \equiv 1 \pmod {4}$. Also, if $v_{2}(z_{3}) = 0$, then again equation \eqref{eq2} yields $1 \equiv b_{1}z_{1}^{2} - b_{1}b_{2}z_{3}^{2} \equiv 0 \pmod {4}$, a contradiction. Consequently, we must have $b{1} \equiv 1 \pmod {4}$.  
\end{enumerate}  
\end{proof}
    %$(v)$ For this case, we focus on solution of (\ref{eq1}) and (\ref{eq2}) modulo $3$. Note that $b_{1}b_{2} \equiv 0 \pmod 9$ implies both $b_{1}$ and $b_{2}$ are divisible by $3$. From (\ref{eq2}), we can then conclude that $v_{3}(z_{j}) \geq 0$ is not a possibility. Similarly, the case of $v_{3}(z_{j}) = -t < 0$ for all $j \in \{1,2,3\}$ can be ruled out from (\ref{eq2}), as that would imply $b_{1} \equiv 0 \pmod 9$, a contradiction. from Lemma \ref{lems1}, we note that the only case remaining is $v_{3}(z_{3}) = -1, v_{3}(z_{j}) \geq 0$ for all $j \in \{1,2\}$. From (\ref{eq2}), this immediately implies $\frac{b_{1}b_{2}}{9}u_{3}^{2} \equiv (m^{4}-1-4m^{2}) \equiv -1 \pmod 3$ which then implies that $\frac{b_{1}}{3} \not \equiv \frac{b_{2}}{3} \pmod 3$, and hence the result follows.   

\section{Proof of Theorem \ref{mainthm}}
First, we notice that the first part of Theorem \ref{mainthm} follows from the discussion of Section \ref{tor-sec} and Section \ref{lower-sec}. Now, we claim that $\langle([p_{i}],[p_{i}])\rangle \subseteq {\rm{Sel}}_{2}(E_{m})$ if $p_{i} \equiv 1 \pmod {4}$, and $\left(\frac{p_{i}}{q_{j}}\right) = \left(\frac{p_{i}}{r_{k}}\right) = 1$. Also, $\langle([-p_{i}],[p_{i}])\rangle \subseteq {\rm{Sel}}_{2}(E_{m})$ if $p_{i} \equiv 3 \pmod 4$, and $\left(\frac{p_{i}}{q_{j}}\right) = \left(\frac{-p_{i}}{r_{k}}\right) = 1$. In both cases, $q_{j}$ and $r_{k}$ vary over all prime factors of $(m^{4}-1-4m^{2})$ and $(m^{4}-1+4m^{2})$ respectively. We only prove the case $\langle([p_{i}],[p_{i}])\rangle \subseteq {\rm{Sel}}_{2}(E_{m})$ for $p_{i} \equiv 1 \pmod {4}$ due to its similarity with the proof of the case $p_{i} \equiv 3 \pmod 4$.

\begin{remark}\label{rems2}
The Jacobian of the intersection of equation \eqref{eq1} and equation \eqref{eq2} for $(b_{1}, b_{2}) = (p_{i}, p_{i})$ with $p_{i} \equiv 1 \pmod 4$ is
\begin{align}
    \begin{pmatrix}
2 \cdot p_{i} \cdot z_{1} & -2 \cdot p_{i} \cdot z_{2} & 0\\
2 \cdot p_{i} \cdot z_{1} & 0 & -2 \cdot {p_{i}}^{2} \cdot z_{3}
\end{pmatrix}
\end{align}
which one can easily observe has rank $2$ modulo $\ell$ whenever $\ell \neq 2, p_{i}$. Hence, except for those $\ell$'s, the geometric genus becomes the same as the arithmetic genus, which is $1$ by the degree-genus formula (cf. \cite{hart}, Section II, Ex. 8.4), and Hasse-Weil bound for a genus one curve can be used for all but those finitely many primes. For $\ell \neq p_{i}$, $\ell \geq 5$, Hasse bound guarantees a non-trivial solution $(z_{1},z_{2},z_{3}) \in \mathbb{F}_{\ell} \times \mathbb{F}_{\ell} \times \mathbb{F}_{\ell}$ of \eqref{eq1} and \eqref{eq2} modulo $\ell$. One can immediately note that all three of $z_{1}, z_{2}, z_{3}$ being zero modulo $\ell$ is not possible as $\ell \neq 2, p_{i}$. Now $z_{1} \equiv z_{2} \equiv 0 \pmod {\ell}$ implies $\ell^{2}$ divides $2(m^{4}-1) $, a contradiction. Similarly, $z_{1} \equiv z_{3} \equiv 0 \pmod \ell$ implies $-(m^{4}-1-4m^{2}) \equiv 0 \pmod \ell \implies l = q_{i}$, contradiction again. By suitably fixing two of $z_{1},z_{2}$ and $z_{3}$, one can now convert equations \eqref{eq1} and \eqref{eq2} into one single equation of one variable with a simple root over $\mathbb{F}_{\ell}$. That common solution can then be lifted to $\mathbb{Q}_{\ell}$ via Hensel's lemma.    
\end{remark}

Now we prove that if $([p_{i}],[p_{i}]) \in \mathbb{Q}(S,2) \times \mathbb{Q}(S,2)$ is such that $p_{i} \equiv 1 \pmod 4$ and $\left(\frac{p_{i}}{q_{j}}\right) = \left(\frac{p_{i}}{r_{k}}\right) = 1$, then $(p_{i}, p_{i}) \in {\rm{Sel}}_{2}(E_{m})$. As already mentioned in Remark \ref{rems2} above, we only prove the existence of the local solution of equation \eqref{eq1} and equation \eqref{eq2} for the primes $\ell = 2,3$, and  $p_{i}$.

\smallskip

For $\ell=2$, we note from Lemma \ref{lems3} that $p_{i} \equiv 1,5 \pmod 8$. For $p_{i} \equiv 1 \pmod 8$, we note that $(u_{1},1,1)$ is a solution to a single variable version of (\ref{eq1}) and (\ref{eq2}) with $v_{2}(z_{j}) < 0$, where $u_{1}^{2} \equiv 1 \pmod 8$. This solution can then be lifted to $\mathbb{Q}_{2}$ via Hensel's lemma. Similarly, if $p_{i} \equiv 5 \pmod 8$, then $(u_{1},1,1)$ is again a solution to a single-variable version of (\ref{eq1}) and (\ref{eq2}) with $v_{2}(z_{j}) = 0$ for all $j \in \{1,2,3\}$ with $u_{1}^{2} \equiv 1 \pmod 8$. This solution again can be lifted to $\mathbb{Q}_{2}$ via Hensel's lemma. 

\smallskip

For $\ell=3$, we first note that $m \equiv 0 \pmod 3$. This implies that neither of $b_{1}$ or $b_{2}$ are divisible by $3$. We now produce simple solutions $(z_{1}, z_{2}, z_{3})$ for equation \eqref{eq1} and equation \eqref{eq2} with two constant components below. In this way, we can treat equation \eqref{eq1} and equation \ref{eq2} as equations of one variable and note that those simple solutions can be lifted to $\mathbb{Q}_{3}$ via Hensel's lemma. For the case $p_{i} \equiv 1 \pmod 3$, choose $v_{3}(z_{j}) < 0$, and $(u_{1},1,1)$ is a solution that can be lifted, where $u_{1}^{2} \equiv 1 \pmod 3$. For the case $p_{i} \equiv 2 \pmod 3$, choose $v_{3}(z_{j}) \geq 0$, and $(u_{1},0,1)$ is a solution that can be lifted, where $u_{1}^{2} \equiv 1 \pmod 3$.
    %\subitem $(iii)$ For the case $b_{1} \equiv -b_{2} \equiv 1 \pmod 3$, choose $v_{3}(z_{j}) \geq 0$, and $(u_{1},1,0)$ is a solution that can be lifted, where $u_{1}^{2} \equiv 1 \pmod 3$.
    %\subitem $(iv)$ For the case $-b_{1} \equiv b_{2} \equiv 1 \pmod 3$, choose $v_{3}(z_{j}) \geq 0$, and $(0,u_{2},1)$ is a solution that can be lifted, where $u_{2}^{2} \equiv 1 \pmod 3$.
    
\smallskip

For $\ell = p_{i}$, taking cue from the previous cases, we produce solutions that can be lifted to $\mathbb{Q}_{p_{i}}$. We note that $b_{1}b_{2} \equiv 0 \pmod {p_{i}^{2}}$, and $\Big(\frac{-b_{1}b_{2}/p_{i}^{2}}{p_{i}}\Big) = 1$, agreeing with Lemma \ref{lems3}. Then under the assumption $v_{p_{i}}(z_{3}) = -1$ and $v_{p_{i}}(z_{j}) \geq 0$ for $j \in \{1,2\}$, $(0,0,u_{3})$ is a simple solution for equation \eqref{eq2} and equation \ref{eq3} that can be lifted to $\mathbb{Q}_{p_{i}}$, where $\left(\frac{-b_{1}b_{2}}{p_{i}^{2}}\right) \cdot u_{3}^{2} \equiv 4m^{2} \pmod {p_{i}}$. The existence of such $u_{3}$ is guaranteed as $\Big(\frac{-b_{1}b_{2}/p_{i}^{2}}{p_{i}}\Big) = 1$.

\smallskip

This concludes the proof as we have shown that solution for equation \eqref{eq1} and equation \eqref{eq2} exists in $\mathbb{Q}_{\ell}$ for every prime number $\ell$ when $(b_{1}, b_{2}) = (p_{i}, p_{i})$ with $p_{i} \equiv 1 \pmod 4$. Because one can trivially observe that a real solution exists for equation \eqref{eq1} and equation \eqref{eq2} too, we can conclude for all $\ell \leq \infty$, local solutions for equation \eqref{eq1} and equation \eqref{eq2} exist. Hence $([p_{i}], [p_{i}]) \in {\rm{Sel}}_{2}(E_{m})$. This completes the proof of Theorem \ref{mainthm}. $\hfill\Box$

\section{Proof of Corollary \ref{maincor}}

It is given that $m^{4} - 1 \pm 4m^{2}$ are prime numbers. Let $m^{4} - 1 - 4m^{2} = q$ and $m^{4} - 1 + 4m^{2} = r$. Noting that $q \equiv -r \equiv 4m^{2} \pmod {p_{i}}$, one can see that $\left(\frac{p_{i}}{q}\right) = \left(\frac{p_{i}}{r}\right) = 1$ for $p_{i} \equiv 1 \pmod 4$ whereas $\left(\frac{p_{j}}{q}\right) = \left(\frac{-p_{j}}{r}\right) = 1$ for $p_{j} \equiv 3 \pmod 4$. Hence, from Lemma \ref{lems3} and Theorem \ref{mainthm}, we can conclude that $\{([p_{i}], [p_{i}]), ([-p_{0}], [p_{0}])\} \in {\rm{Sel}}_{2}(E_{m})$ where $p_{i} \equiv 1 \pmod 4$, and $p_{0} \equiv 3 \pmod 4$. 

\smallskip

We now note that $q \equiv r  \equiv 7 \pmod 8$, and consequently, $\left(\frac{2}{q}\right) = \left(\frac{2}{r}\right) = 1$. Hence, from Lemma \ref{lems3}, $([-q], [1])$ and $([1], [r])$ satisfy necessary properties for potential elements in ${\rm{Sel}}_{2}(E_{m})$. We now prove that $(-q,1) \in {\rm{Sel}}_{2}(E_{m})$. This will be sufficient to show that $(1,r) \in {\rm{Sel}}_{2}(E_{m})$. This is because $(-q,r) \in {\rm{Sel}}_{2}(E_{m})$ always, due to being the image of the torsion point $(4m^{2}, 0)$, and hence $(1,r)  = (-q,1) \cdot (-q,r) \in [(-q,1)]$.

\smallskip

An approach similar to Remark \ref{rems2} shows that it is enough to prove that for $(-q,1)$, equation \eqref{eq1} and equation \eqref{eq2} have solutions over $\mathbb{Q}_{l}$ for $l = 2,3$ and $q$. We first note that $q \equiv -1 \pmod 8$, and without loss of generality one can assume $m \equiv 0 \pmod 3$ since both $m + 1$ and $m - 1$ are prime numbers, $q = m^{4}-1-4m^{2} \equiv -1 \pmod 3$ also. This, in turn, shows $(u_{1},1,1)$ is a solution of equation \eqref{eq1} and equation \eqref{eq2} with $v_{2}(z_{i}) = -t < 0$, that can be lifted to $\mathbb{Q}_{2}$ via Hensel's lemma, where $u_{1} \not \equiv 0 \pmod 2$. In a similar way, one can show that $(u_{1},1,1)$ is a solution of equation \eqref{eq1} and equation \eqref{eq2} with $v_{2}(z_{i}) = -t < 0$, that can be lifted to $\mathbb{Q}_{3}$, where $u_{1} \not \equiv 0 \pmod 3$ too. For $l = q$, we note that $(0,u_{2},0)$ with $u_{2}^{2} \equiv 8m^{2} \pmod q$ is a solution to equation \eqref{eq1} and equation \eqref{eq3}, that can be lifted to $\mathbb{Q}_{q}$. Hence, we can conclude $([-q], [1]) \in {\rm{Sel}}_{2}(E_{m})$.

\smallskip

We have proved that $\langle{([p_{i}], [p_{i}]), ([-p_{0}], [p_{0}]), ([-q], [1])\rangle} \subseteq {\rm{Sel}}_{2}(E_{m})$. In fact, Lemma \ref{lems3} asserts that the inclusion is not strict, i.e., $\langle{([p_{i}], [p_{i}]), ([-p_{0}], [p_{0}]), ([-q], [1])\rangle} = {\rm{Sel}}_{2}(E_{m}).$ Noting that $p_{i}$ varies over all $4k+1$ prime factors of $m^{4}-1$, and $p_{0}$ is the only $4k+3$ factor of $m^{4}-1$, now the result follows. $\hfill\Box$

\begin{ack}
We sincerely thank the anonymous referees for their careful reading of the manuscript and pointing out an issue in the computation of the Mordell-Weil rank of $E_{m}$ in an earlier version of the paper. Their insightful remarks helped us improving the proof of the main theorem. We are grateful to our respective institutions for providing excellent facilities to carry out this research. The research of the first author is supported by UGC (Grant no. 211610060298) and he gratefully acknowledges that. The computations in Table 1 are carried out using SageMath software.
\end{ack}

\end{document}